\documentclass [11pt,oneside,a4paper,mathscr]{amsart}

\usepackage{amscd,amsmath,amssymb,euscript}
\usepackage[frame,cmtip,curve,arrow,matrix,line,graph]{xy}
\usepackage{tikz-cd}
\usepackage{bigints}
\usetikzlibrary{matrix}
\usepackage{hyperref}
\usepackage{stmaryrd}

\title[Generators for Hall algebras]{Generators for Hall algebras of surfaces}
\author{Tudor P\u adurariu}
\address{Department of Mathematics, Columbia University, 
2990 Broadway, New York, NY 10027}
\email{tgp2109@columbia.edu}

\newtheorem{thm}{Theorem}[section]
\newtheorem{cor}[thm]{Corollary}
\newtheorem{prop}[thm]{Proposition}

\theoremstyle{definition}

\newtheorem{thm*}[thm]{Theorem$^*$}

\newcommand{\comment}[1]{}

\renewcommand{\leq}{\leqslant}
\renewcommand{\geq}{\geqslant}

\newcommand{\F}{\mathcal{F}}
\newcommand{\PP}{\operatorname{\mathfrak P}}

\newcommand{\A}{\mathcal{A}}

\newcommand{\B}{\mathcal{B}}

\newcommand{\X}{\mathcal{X}}
\newcommand{\MM}{\mathfrak{M}}

\newcommand{\C}{\mathbb{C}}

\newcommand{\ee}{\underline{e}}
\newcommand{\dd}{\underline{d}}

\begin{document}
\maketitle

\begin{abstract}
For a smooth surface $S$, Porta--Sala defined a categorical Hall algebra generalizing previous work in K-theory of Zhao and Kapranov--Vasserot. We construct semi-orthogonal decompositions for categorical Hall algebras of points on $S$. We refine these decompositions in K-theory for a topological K-theoretic Hall algebra.
\end{abstract}

\section{Introduction}

\subsection{Hall algebras of surfaces}
Let $S$ be a smooth surface over $\mathbb{C}$.
For an algebraic class $\beta\in H^2(S,\mathbb{Z})\oplus H^4(S,\mathbb{Z})$, let $\mathfrak{M}_\beta$ be the (derived) moduli stack of coherent sheaves on $S$ with support class $\beta$. For classes $\beta$ and $\gamma$, there are maps 
\begin{equation}\label{maps}
\mathfrak{M}_{\beta}\times \mathfrak{M}_{\gamma}\xleftarrow{q_{\beta, \gamma}}\mathfrak{M}_{\beta, \gamma}\xrightarrow{p_{\beta,\gamma}}\mathfrak{M}_{\beta+\gamma},
\end{equation}
where $\mathfrak{M}_{\beta, \gamma}$ is the corresponding stack of extensions. The map $q_{\beta,\gamma}$ is quasi-smooth and the map $p_{\beta,\gamma}$ is proper, so they induce functors 
\[m_{\beta, \gamma}:=p_{\beta,\gamma*}q^*_{\beta,\gamma}:
D^b\left(\mathfrak{M}_{\beta}\right)\otimes D^b\left(\mathfrak{M}_{\gamma}\right)\to
D^b\left(\mathfrak{M}_{\beta+\gamma}\right).\]
Porta--Sala \cite{PS} showed that the
category $\bigoplus_{\beta} D^b\left(\mathfrak{M}_{\beta}\right)$ is monoidal with respect to the functors $m_{\beta,\gamma}$. Taking the Grothendieck group of this category, we obtain the K-theoretic Hall algebra (KHA) of a surface which has been studied by Zhao \cite{Z} for sheaves of dimension zero and by Kapranov--Vasserot \cite{KV}. 

Categorical/ K-theoretic Hall algebras for quivers with potential are local version of these categories/ algebras. Particular cases of (equivariant) KHAs of quivers with potentials, namely preprojective KHAs, are expected to be positive parts of quantum affine groups constructed by Okounkov--Smirnov \cite{os}. 

It is interesting to see whether KHAs of surfaces have properties similar to (positive parts of) quantum groups, for example if they satisfy a PBW theorem, if they are deformations of (the universal enveloping algebra of) a K-theoretic Lie algebra associated to the surface, or if they can be doubled to a Hopf algebra.



\subsection{Semi-orthogonal decompositions of the HA}
Let $S$ be a smooth surface over $\mathbb{C}$.
For $d\in \mathbb{N}$, let $\mathfrak{M}_d$ be the moduli stack of dimension zero sheaves of length $d$ on $S$. Let $\text{HA}(S):=\bigoplus_{d\in\mathbb{N}}D^b\left(\mathfrak{M}_d\right)$.
In Subsection \ref{M}, we define categories $\mathbb{M}(d)_{w}\subset D^b\left(\mathfrak{M}_d\right)_w$ for $w\in\mathbb{Z}$. Consider the map \begin{equation}\label{ch}
    \pi_d: \mathfrak{M}_d\to M_d:=\text{Sym}^d(S).
\end{equation} 
Let $V^d_w$ be the set of partitions $A=(d_i, w_i)_{i=1}^k$ of $(d, w)$ with \[\frac{w_1}{d_1}>\cdots>\frac{w_k}{d_k}.\] For $A\in V^d_w,$ let $\mathbb{M}_A:=\otimes_{i=1}^k \mathbb{M}\left(d_i\right)_{w_i}$.

\begin{thm}\label{thm}
 There is a semi-orthogonal decomposition
\[D^b\left(\mathfrak{M}_d\right)_w=\Big\langle 
\mathbb{M}_A\Big\rangle,\]
where the right hand side is after all partitions $A\in V^d_w$.
The order of categories is as in Subsection \ref{compadm}. The semi-orthogonal decomposition holds over $M_d$ in the following sense. Let $A<B$ be two partitions and let $\mathcal{F}\in \mathbb{M}_A$ and $\mathcal{G}\in \mathbb{M}_B$. Then \[R\pi_{d*}\left(R\mathcal{H}om_{\mathfrak{M}_d}(\mathcal{F}, \mathcal{G})\right)=0.\]
\end{thm}


The proof of Theorem \ref{thm} is as follows. The map $\pi_d$ is analytically (and formally) locally on $\text{Sym}^d(S)$ described using the moduli stack of dimension $d$ sheaves on an open subset $U\subset \mathbb{A}^2_{\mathbb{C}}$. The categorical Hall algebra of $\mathbb{A}^2_{\mathbb{C}}$ is equivalent to the preprojective Hall algebra of $J$, the Jordan quiver, and also equivalent to the Hall algebra of the quiver $\widetilde{J}$ with three loops $x, y, z$ and potential $\widetilde{W}:=xyz-xzy$. We glue the decompositions of $\text{HA}\left(\widetilde{J}, \widetilde{W}\right)$ from \cite{P} to decompositions of $\text{HA}(S)$.

It is interesting to see if Theorem \ref{thm} holds for Hall algebras of semistable sheaves of support $\beta\in H^2(S, \mathbb{Z})\oplus H^4(S, \mathbb{Z})$ on $S$. By a result of Toda \cite{T2}, the moduli stacks \begin{equation}\label{toad}
    \pi_\beta: \mathfrak{M}_{\beta}^{\text{ss}}\to M^{\text{ss}}_{\beta}
\end{equation} admit descriptions using quivers of potential $(Q,W)$ over $M^{\text{ss}}_{\beta}$, and we can thus try to use the results in \cite{P}. In loc. cit., the semi-orthogonal decompositions for $\text{HA}(Q,W)$ depend on certain Weyl-invariant weights $\delta_{A}$ for $A=(d_i, w_i)_{i=1}^k$ a partition of $(d,w)$. 
The main difficulty in proving an analogue of Theorem \ref{thm} for $\mathfrak{M}^{\text{ss}}_\beta$ is that the local weights $\delta_A$ cannot be glued to a global $\delta_A\in \text{Pic}\left(\mathfrak{M}^{\text{ss}}_{\beta, n}\right)_{\mathbb{R}}$, so it is not clear how to define $\mathbb{M}_A\subset D^b\left(\mathfrak{M}_n\right)$.  
In the case of $\widetilde{J}$, the Weyl-invariant weights $\delta$ are multiples to each other and the categories of generators $\mathbb{M}$ do not depend on $\delta$. 

\subsection{A PBW theorem for a topological KHA}

Denote by $\text{KHA}(S)$ the Grothendieck group of $\text{HA}(S)$. 
Let \[\text{Sh}(S):=\bigoplus_{n\geq 1}\left(K_0(S^n)(z_1,\cdots, z_n)\right)^{\mathfrak{S}_n}\] be the algebra considered by Negu\c{t} \cite{N}, Zhao \cite{Z}.
Zhao in loc. cit. constructed an algebra morphism \[\Phi_S: \text{KHA}(S)\to \text{Sh}(S).\] There are analogous constructions for topological K-theory by applying Blanc's topological K-theory \cite{Bl} to the Porta--Sala monoidal category. Zhao's construction applies to construct an algebra morphism
\[\Phi_S^{\text{top}}: \text{KHA}^{\text{top}}(S)\to \text{Sh}^{\text{top}}(S).\]
Denote by $\text{KHA}'^{\text{top}}(S)$ the image of $\Phi$ and by $\text{KHA}'^{\text{top}}(S)_{d,w}$ its graded $(d,w)$-part. 
In Subsection \ref{pbwth}, we define subspaces $P(d)_w\subset \text{KHA}'^{\text{top}}(S)_{d,w}\otimes \mathbb{Q}$. Using a local argument and \cite[Proposition 5.5]{P} for $\text{HA}\left(\widetilde{J}, \widetilde{W}\right)$, we prove a PBW-type theorem for $\text{KHA}'^{\text{top}}$. Let $U^d_w$ be the set of partitions $A=(d_i, w_i)_{i=1}^k$ of $(d,w)$ 
such that \[\frac{w_1}{d_1}=\cdots=\frac{w_k}{d_k}.\] 

\begin{thm}\label{thm2}
Assume $S$ is a smooth surface with $H^1(S,\mathbb{Q})=0$.
Let $d\in \mathbb{N}$ and $w\in\mathbb{Z}$. There is a decomposition 
\begin{equation}\label{kk}
    K_0'^{\text{top}}\left(\mathbb{M}(d)_w\right)_{\mathbb{Q}}\cong
\bigoplus_{U^d_w} \bigotimes_{i=1}^k\text{Sym}^{\ell_i}\left(P(d_i)_{w_i}\right),\end{equation} where the right hand side is after all partitions in $U^d_w$ with $\ell_i$ summands $(d_i,w_i)$ for $1\leq i\leq s$ such that $d_1<\cdots<d_s$.
\end{thm}

Theorem \ref{thm} and Theorem \ref{thm2} imply a PBW-type decomposition for $\text{KHA}'^{\text{top}}(S)_{\mathbb{Q}}$ for $S$ as above, for example it implies that $\text{KHA}'^{\text{top}}(S)_{\mathbb{Q}}$ is generated by the elements $x_{d_1,w_1}\cdots x_{d_k, w_k}$ for $x_{d_i, w_i}\in P(d_i)_{w_i}$ with 
$\frac{w_1}{d_1}\geq \cdots\geq\frac{w_k}{d_k}$.
It would be interesting to obtain a version of Theorem \ref{thm2} for the (algebraic) KHA. 
This would require an analogue of Theorem \ref{thm51} for algebraic K-theory.  


\subsection{Previous work on Hall algebras of surfaces}

The KHA of $\mathbb{A}^2_{\mathbb{C}}$ was studied by Schiffmann--Vasserot \cite{SV} and it is the positive part of $U_q^{>}\left(\widehat{\widehat{\mathfrak{gl}_1}}\right)$. For a curve $C$, the Hall algebra of (sheaves with compact support on) $T^*C$ is the Hall algebras for Higgs bundles on $C$ studied by Sala--Schiffmann \cite{SS}, Minets \cite{M}.

Toda has studied the relation between Hall algebras and Donaldson--Thomas theory of local surfaces $\text{Tot}_S\left(\omega\right)$ in \cite{T}, \cite{T3}, \cite{T4}, \cite{T5}.

Kapranov--Vasserot \cite{KV} proved a PBW theorem for CoHAs of dimension zero sheaves on surfaces using factorization algebras.

\subsection{Structure of the paper} In Section \ref{2}, we review semi-orthogonal decompositions for categorical Hall algebras of quivers with potential. 
In Section \ref{comparison}, we compare some of these categories with categorical Hall algebras constructed from quotients of path algebras. A particular case of these quotients is the preprojective algebra of a quiver, but our results are more general and cover all the quotients that appear in local descriptions of the map \eqref{toad}. In particular, we obtain semi-orthogonal decomposition for $D^b\left(\widehat{\mathfrak{M}^{\text{ss}}_\beta}\right)$, where $\widehat{\mathfrak{M}^{\text{ss}}_\beta}$ is the formal completion of $\mathfrak{M}^{\text{ss}}_\beta$ along $\pi_{\beta}^{-1}(p)$ for any $p\in M^{\text{ss}}_\beta$, see \cite[Lemma 5.4.1, Section 7.4]{T}.
As a particular case of semi-orthogonal decompositions of preprojective algebras, we prove Theorem \ref{thm} for $\mathbb{A}^2_{\mathbb{C}}$. 
In Section \ref{5}, we prove Theorem \ref{thm}. In Section \ref{6}, we prove Theorem \ref{thm2}.

\subsection{Acknowledgements.} 
I thank Francesco Sala, Yukinobu Toda, and the referee for useful comments and suggestions. 
I thank the Institute of Advanced Studies for support during the preparation of the paper. This material is based upon work supported by the National Science Foundation under Grant No. DMS-1926686.

\subsection{Notations and conventions.} All stacks considered are over $\mathbb{C}$. All surfaces considered are smooth.

For $\X$ a quasi-smooth stack, let $D^b(\X)$ be the category of bounded complexes of coherent sheaves on $\X$ and let $\text{Ind}\,D^b(\X)$ be the Ind-completion of $D^b(\X)$ \cite{G}, see \cite[Section 3.1]{T} for a brief reminder on these topics. For a brief reminder of functoriality of the quasi-smooth stacks used in this paper, see \cite[Section 4.2]{PS}, \cite[Section 3.1]{T}. We denote by $\text{Hom}$ the morphism spaces in these categories and by $\mathcal{H}om$ the internal Hom-spaces in these categories.
For the stacks considered in the paper, there is a dualizing complex $\omega_X$ such that the dualizing functor is an equivalence
\[\mathbb{D}_\X(\F):=R\mathcal{H}om_\X(\F,\omega_\X): D^b(\X)\xrightarrow{\sim} D^b(\X)^{\text{op}},\]
see \cite[Section 2.2]{HL}.

We denote by $\text{MF}^{\text{gr}}$  and by $\text{MF}^{\text{gr}}_{\text{qcoh}}$ the categories of coherent and quasi-coherent matrix factorizations, respectively, see \cite[Section 2.2]{T} for definitions and functoriality of these categories. These categories also admit dualizing functors $\mathbb{D}$ induced by the canonical sheaf of the ambient variety and they satisfy a Thom-Sebastiani theorem, see \cite{Pr}, \cite[Section 3 and Section 5.1]{BFK}, \cite[Section 2.5]{EP}.

The categories considered are dg and we denote by $\otimes$ the product of dg categories \cite[Subsections 2.2 and 2.3]{K}.
For stacks $\X$ and $\mathcal{Y}$ considered in this paper, we have that $D^b\left(\X\times\mathcal{Y}\right)\cong D^b(\X)\otimes D^b\left(\mathcal{Y}\right)$ by \cite{BZFN}. 

We use the notation $\widehat{\X}$ for formal completions of $\X$ along a specified substack. For a morphism $f$, let $\mathbb{L}_f$ be its cotangent complex.
For a stack $\X\to \text{Spec}\,\mathbb{C}$, denote by $\mathbb{L}_\X$ its cotangent complex.

\section{Hall algebras of quivers and dimensional reduction}\label{2}

\subsection{Preliminaries}

\subsubsection{} Let $(Q,W)$ be a symmetric quiver with potential.
Let $I$ and $E$ be the sets of vertices and edges, respectively, of $Q$.
For $d\in\mathbb{N}^I$, consider the stack \[\X(d)=R(d)/G(d)\] of representations of $Q$ of dimension $d$. Fix maximal torus and Borel subgroups $T(d)\subset B(d)\subset G(d)$. We use the convention that the weights of the Lie algebra of $B(d)$ are negative; it determines a dominant chamber of weights of $G(d)$. Let $M$ be the weight space of $G(d)$, let $M_{\mathbb{R}}:=M\otimes_{\mathbb{Z}}\mathbb{R}$, and let $M^+\subset M$ and $M^+_{\mathbb{R}}\subset M_{\mathbb{R}}$ be the dominant chambers. When we want to emphasize the dimension vector, we write $M(d)$ etc. Denote by $N$ the coweight lattice of $G(d)$ and by $N_{\mathbb{R}}:=N\otimes_{\mathbb{Z}}\mathbb{R}$. Let $\langle\,,\,\rangle$ be the natural pairing between $N_{\mathbb{R}}$ and $M_{\mathbb{R}}$. 

Let $\mathfrak{S}_d$ be the Weyl group of $G(d)$. For $\chi\in M(d)^+$, let $\Gamma(\chi)$ be the representation of $G(d)$ of highest weight $\chi$.


Consider the regular function $\text{Tr}\,W:\X(d)\to\mathbb{A}^1_\mathbb{C}$.
Let $\X(d)_0:=\X(d)\times_{\mathbb{A}^1_\mathbb{C}} 0$ be its (derived) zero fiber.

Denote by $SG(d):=\text{ker}\left(\text{det}: G(d)\to\mathbb{C}^*\right)$.

\subsubsection{}\label{action}
Assume there is an extra $\mathbb{C}^*$-action on $R(d)$ which commutes with the action of $G(d)$ on $R(d)$ and such that $\text{Tr}\,W$ is of weight $2$. 
Consider the category of graded matrix factorizations $\text{MF}^{\text{gr}}(\X(d), W)$ for the regular function $\text{Tr}\,W$ \cite[Section 2.2]{T}. It is equivalent to the graded category of singularities $D^{\text{gr}}_{\text{sg}}(\X(d)_0)$.

\subsubsection{}\label{weight}
The action of $z\cdot\text{Id}\subset G(d)$ on $R(d)$ is trivial. Let $D^b(\X(d))_w$ be the category of complexes on which $z\cdot\text{Id}$ acts with weight $w$. We have an orthogonal decomposition $D^b(\X(d))=\bigoplus_{w\in\mathbb{Z}} D^b(\X(d))_w$.

More generally, for a stack $\X$ with a trivial action of a torus $T$, there is a decomposition 
\[D^b(\X)=\bigoplus_{\chi\in X(T)}D^b(\X)_\chi,\]
where $X(T)$ is the character lattice of $T$. For $\chi\in X(T)$,
denote the projection functor by 
\[\beta_\chi:D^b(\X)\to D^b(\X)_\chi.\]
There are analogous definitions and decompositions for categories of quasi-coherent sheaves, for $\text{MF}^{\text{gr}}$, and for $\text{MF}^{\text{gr}}_{\text{qcoh}}$.

\subsubsection{} 
Denote by $\beta^i_j$ the simple roots of $G(d)$ for $i\in I$ and $1\leq j\leq d_i$, where $d=(d^i)_{i\in I}\in \mathbb{N}^I$, and by
$\rho$ half the sum of positive roots of $G(d)$. We denote by $1_d:=z\cdot\text{Id}$ the diagonal cocharacter of $G(d)$. 
Consider the real weights 
\begin{align*}
    \nu_d&:=\sum_{\substack{i\in I\\ j\leq d_i}} \beta^i_j,\\
    \tau_d&:=\frac{\nu_d}{\langle 1_d, \nu_d\rangle}.
\end{align*}

\subsubsection{}\label{nm}
For a cocharacter $\lambda:\C^*\to SG(d)$, consider the maps of fixed and attracting loci
\begin{equation}\label{e}
\X(d)^\lambda \xleftarrow{q_\lambda}\X(d)^{\lambda\geq 0}\xrightarrow{p_\lambda}\X(d).\end{equation}
The map $q_\lambda$ is an affine bundle map, in particular it is smooth, and the map $p_\lambda$ is a proper map.
We say that two cocharacters $\lambda$ and $\lambda'$ are equivalent and write $\lambda\sim\lambda'$ if $\lambda$ and $\lambda'$ have the same fixed and attracting stacks as above. 
For a cocharacter $\lambda$ of $SG(d)$, define \begin{equation}\label{weight2}
    n_\lambda=-\big\langle \lambda, \det\mathbb{L}_{\X(d)}^{\lambda\leq 0}|_{\X(d)^\lambda}\big\rangle.
\end{equation}
Consider the natural inclusion $\alpha:0/G(d)\hookrightarrow \X(d)$. We use the notation $\mathcal{F}|_0:=\alpha^*(\mathcal{F})$ for the restriction of $\mathcal{F}$ at the origin.
For a cocharacter $\lambda$, we have an associated ordered partition $d_1+\cdots+d_k=d$ such that $\X(d)^\lambda\cong\times_{i=1}^k\X(d_i)$; the order is induced by the choice of $B(d)\subset G(d)$. Define the length $\ell(\lambda):=k$.

Define the polytope $\mathbb{W}$ by
\[\mathbb{W}:= \text{span}\,[0,\beta]\oplus \mathbb{R}\tau_d\subset M(d)_{\mathbb{R}},\]
where the span is after all weights $\beta$ of $R(d)$.
Let $\delta\in M_{\mathbb{R}}^{\mathfrak{S}_d}$ and define $\mathbb{N}(d; \delta)$ as the subcategory of $D^b(\X(d))$ with complexes $\mathcal{F}$ such that 
\[-\frac{n_\lambda}{2}+\langle \lambda, \delta\rangle\leq \langle \lambda, \mathcal{F}|_0\rangle \leq \frac{n_\lambda}{2}+\langle \lambda, \delta\rangle,\]
for all cocharacters $\lambda$ of $SG(d)$. 
Alternatively, it is the subcategory of $D^b(\X(d))$ generated by $\mathcal{O}_{\X(d)}\otimes \Gamma(\chi)$ for $\chi$ a dominant of weight of $G(d)$ such that
\[\chi+\rho+\delta\in \frac{1}{2}\mathbb{W}.\]

Let $\mathbb{M}(d; \delta)$ be $\text{MF}^{\text{gr}}\left(\mathbb{N}(d; \delta), W\right)$, the category of matrix factorizations with factors in $\mathbb{N}(d; \delta)$. 
For an ordered partition $(d_1,\cdots, d_k)$ of $d$, fix an antidominant cocharacter $\lambda_{d_1,\cdots,d_k}$ which induces the maps
\[\times_{i=1}^k\X(d_i)\cong\X^\lambda\xleftarrow{q_\lambda}\X^{\lambda\geq 0}\xrightarrow{p_\lambda}\X.\] 
The multiplication is induced by the functor $p_{\lambda*}q_\lambda^*$. We may drop the subscript $\lambda$ in the functors $p_*$ and $q^*$ when the cocharacter $\lambda$ is clear.

\subsubsection{}\label{compa}
Let $\underline{e}=(e_i)_{i=1}^l$ and $\dd=(d_i)_{i=1}^k$ be two partitions of $d\in \mathbb{N}^I$. We write $\ee\geq\dd$
if there exist integers \[a_0=0< a_1<\cdots<a_{k-1}\leq a_k=l\] such that for any $0\leq j\leq k-1$, we have
\[\sum_{i=a_{j}+1}^{a_{j+1}} e_i=d_{j+1}.\]
There is a similarly defined order on pairs $(d,w)\in\mathbb{N}^I\times\mathbb{Z}$.


\subsubsection{}\label{id} Let $(d_i)_{i=1}^k$ be a partition of $d$. There is an identification \[\bigoplus_{i=1}^k M(d_i)\cong M(d),\] where the simple roots $\beta^i_j$ in $M(d_1)$ correspond to the first $d_1$ simple roots $\beta^i_j$ of $d$ etc.

\subsubsection{}\label{sod4}
Let $I$ be a set. Assume there is a set $O\subset I\times I$ such that for any $i, j\in I$ we have that $(i,j)\in O$, or $(j,i)\in O$, or both $(i,j)\in O$ and $(j,i)\in O$. If $(i,j)\in O$, write $i>j$.

Let $\mathbb{T}$ be a triangulated category. We will construct semi-orthogonal decompositions
\[\mathbb{T}=\langle \mathbb{A}_i\rangle,\] 
by subcategories $\mathbb{A}_i$ with $i\in I$
such that for any $i,j\in I$ with $i>j$ and objects $\mathcal{A}_i\in\mathbb{A}_i$, $\mathcal{A}_j\in\mathbb{A}_j$, we have that:
\[\text{RHom}^\cdot_{\mathbb{T}}(\mathcal{A}_j,\mathcal{A}_i)=0.\]
If there exists a minimal element $o$ in $I$, then the inclusion $\mathbb{A}_o\hookrightarrow\mathbb{T}$ admits a right adjoint $R:\mathbb{T}\to \mathbb{A}_o$.

\subsection{Semi-orthogonal decompositions of categorical Hall algebras of quivers with potential}

In this Section, we recall the semi-orthogonal decomposition of categorical Hall algebras of quivers with potential from \cite[Theorem 1.1]{P2}. We first need to discuss some preliminary notions. We fix a dimension vector $d\in\mathbb{N}^I$. Denote by $Q^d$ the subquiver of $Q$ with vertices $i\in I$ such that $d_i>0$.
We assume that $Q=Q^d$.

\subsubsection{}
For $\chi$ a weight in $M_{\mathbb{R}}$, define its $r$-invariant to be the smallest nonnegative real number $r$ such that
$\chi\in r\mathbb{W}.$ There exists a maximal antidominant cocharacter $\lambda$ such that \[\langle \lambda, \chi\rangle=-r\langle \lambda, N^{\lambda>0}\rangle.\]
The $r$-invariant is a measurement of how close a weight is from the polytope $\frac{1}{2}\mathbb{W}$ used in the definition of the categories $\mathbb{M}(d; \delta)$.

\subsubsection{}\label{tree}

We define a tree $\mathcal{T}$ of partitions which will help us keep track of decompositions of weights in $M(d)_\mathbb{R}$ and, in particular, will specify the order in the semi-orthogonal decomposition from Theorem \ref{qp}.

The tree $\mathcal{T}=(\mathcal{I}, \mathcal{E})$ is defined as follows: the set $\mathcal{I}$ of vertices has elements corresponding to partitions $\dd$ of any dimension vector $d\in\mathbb{N}^I$ and the set $\mathcal{E}$ has edges from $\dd=(d_i)_{i=1}^k$ to $\ee$ if $\ee$ is a partition of $d_i\neq d$ for some $1\leq i\leq k$. 

A set $T\subset \mathcal{I}$ is called a \textit{tree of partitions} if it is finite and has a unique element $v\in T$ with in-degree zero.
Let $\Omega\subset T$ be the set of vertices with out-degree zero. Define the Levi group associated to $T$:
$$L(T):=\prod_{\ee\in \Omega} L(\ee).$$

\subsubsection{}

We recall the discussion from \cite[Subsection 3.1.2]{P}. 
Let $\chi\in M^+_\mathbb{R}$. Then there exists a tree $T$ of antidominant cocharacters with associated Levi group $L$, see Subsection \ref{tree}, such that there exists $\psi\in \frac{1}{2}\mathbb{W}$ with
\begin{equation}\label{godown5}
    \chi=-\sum_{j\in T} r_{j}N_j+\psi,\end{equation}
and if $i, j\in T$ are vertices such that there exists a path from $i$ to $j$, then $r_{i}> r_{j}> \frac{1}{2}$.
For $d_a$ a summand of a partition of $d$, denote by $M(d_a)\subset M(d)$ the subspace as in the decomposition from Subsection \ref{id}. 
The weight $N_j$ is defined as follows. Assume that $j$ is a partition of a dimension vector $d_a\in \mathbb{N}^I$. Let $\mathcal{W}_{j}\subset \mathcal{W}$ be the set of weights
$\beta$ of $R(d_a)$ 
with $\langle \lambda_j, \beta\rangle>0$. Define \begin{equation}\label{nj}
    N_j:=\sum_{\mathcal{W}_{j}}\beta.\end{equation}
We explain the idea behind the decomposition \eqref{godown5}, for more details see \cite[Subsection 3.1.2]{P}. If $\chi$ has $r$-invariant $r$, we locate a face of the polytope $r\mathbb{W}$ on which $\chi$ lies. Assume this face is given by the cocharacter $\lambda$ and that $\chi$ lies in the interior of this face. We then write 
\[\chi=-rN^{\lambda>0}+\psi,\] where $\psi$ is a weight with $r$-invariant $s<r$. Repeating the above procedure for the weight $\psi$ and letting $r:=r_1$ and $N_1:=N^{\lambda>0}$, we obtain the desired decomposition.

\subsubsection{}
\label{admissible}

We assume that $Q$ is connected and that $Q$ is not $Q^o$. Fix $\delta\in M^{\mathfrak{S}_d}_{\mathbb{R}}$.
Let $\chi\in M^+$ and consider the standard form \eqref{godown5}:
\[\chi+\rho+\delta=-\sum_{j\in T} r_jN_j+\psi,\]
for $\psi\in \frac{1}{2}\mathbb{W}$.
Let $L\cong \times_{i=1}^k G(d_i)$. Write $\chi=\sum_{i=1}^k \chi_i$ with $\chi_i\in M(d_i)^+$ for $1\leq i\leq k$ and consider the associated partition \[A=A_\chi:=(d_i,w_i)_{i=1}^k,\] where $w_i=\langle 1_{d_i}, \chi_i\rangle$. Let $S^d_w(\delta)$ be the set of partitions $A$ for which there exists $\chi\in M^+$ such that $A_\chi=A$. Let $T^d_w(\delta)$ be the set of partitions $A=(d_i,w_i)_{i=1}^k$ for which there exist $\chi=\sum_{i=1}^k\chi_i\in M^+$
with $\chi_i\in M(d_i)^+$ for $1\leq i\leq k$, $w_i=\langle 1_{d_i},$ such that 
\[\chi+\rho+\delta=-\frac{1}{2}N^{\lambda_{\dd}<0}+\psi\] for a weight $\psi$ in the interior of $\frac{1}{2}\mathbb{W}$.
To any such partition $A$ with cocharacter $\lambda$, associate the weight
\begin{equation}\label{weightA}
\chi_A:=-\sum_{j\in T} r_jN_j-\rho^{\lambda<0}-\delta.
\end{equation}
For $1\leq i\leq k$, consider weights $\delta_{Ai}\in M(d_i)_{\mathbb{R}}^{\mathfrak{S}_{d_i}}$ defined by
\begin{equation}\label{deltaai}
    -\chi_A=\sum_{i=1}^k \delta_{Ai}\text{ in }M(d)_{\mathbb{R}}\cong \bigoplus_{i=1}^k M(d_i)_{\mathbb{R}}.
    \end{equation}

\subsubsection{}\label{swe}
In the current and next Subsections, we explain the order in the semi-orthogonal decomposition from Theorem \ref{qp}. 

Assume first that $Q=Q^o$. Then $\delta$ is a multiple of $\tau_d$. Let $S^d_w(\delta)$ be the set of partitions $(1,w_i)_{i=1}^d$ of $(d,w)\in\mathbb{N}\times\mathbb{Z}$ with $w_1\geq\cdots\geq w_d$. 

Assume that $Q$ is a disconnected quiver and let $\delta\in M(d)_{\mathbb{R}}^{\mathfrak{S}_d}$.
If $Q$ is a disjoint union of connected quivers $Q_j$ for $j\in J$, write $d_j$ and $\delta_j$ for the corresponding dimension vector and Weyl invariant weight of $Q_j$ for $j\in J$. Let $C$ be the set of partitions $(w_j)_{j\in J}$ of $w$.
Let 
\[S^d_w(\delta):=\bigsqcup_C\left(\times_{j\in J} S^{d_j}_{w_j}(\delta_j)\right).\]

\subsubsection{}\label{compadm}
We explain how to compare partitions in $V^d_w(\delta)$. Assume first that $Q$ is connected and that $Q$ is not $Q^o$. Consider two partitions $A=(d_i, w_i)_{i=1}^k$ and $B=(e_i, v_i)_{i=1}^l$ in $S^d_w(\delta)$. Let $\chi_A$ and $\chi_B$ be two weights with associated sets $A$ and $B$:
\begin{align*}
    \chi_A+\rho+\delta&:=-\sum_{j\in T_A}r_{A,k}N_{A,k}+\psi_A,\\
    \chi_B+\rho+\delta&:=-\sum_{j\in T_B}r_{B,k}N_{B,k}+\psi_B.
\end{align*}
The set $R^d_w(\delta)\subset S^d_w(\delta)\times S^d_w(\delta)$ contains pairs $(A,B)$ for which there exists $c\geq 1$ such that $r_{A,c}>r_{B,c}$ and $r_{A,i}=r_{B,i}$ for $i<c$, or for which there exists $c\geq 1$ such that $r_{A,i}=r_{B,i}$ for $i\leq c$, $\lambda_{A,i}=\lambda_{B,i}$ for $i<c$, and $\lambda_{A,c}> \lambda_{B,c}$, or with $A=B$.

Let $O^d_w(\delta):=S^d_w(\delta)\times S^d_w(\delta)\setminus R^d_w(\delta)$.

For $Q^o$, let $R^d_w(\delta)=\{(A,A)|A\in S^d_w(\delta)\}$, and let $O^d_w(\delta):=S^d_w(\delta)\times S^d_w(\delta)\setminus R^d_w(\delta)$.

Assume $Q$ is a disconnected quiver. We continue with the notation from the previous Subsection.
Consider the set $R^{d_j}_{w_j}(\delta_j)\subset S^{d_j}_{w_j}(\delta_j)\times S^{d_j}_{w_j}(\delta_j)$ for the quiver $Q_j$, let \[R^d_w(\delta):=\bigsqcup_{C} \left(\times_{j\in J} R^{d_j}_{w_j}(\delta_j)\right)\subset S^d_w(\delta)\times S^d_w(\delta),\] and let $O^d_w(\delta):=S^d_w(\delta)\times S^d_w(\delta)\setminus R^d_w(\delta)$.

For a partition $\dd=(d_i)_{i=1}^k$ of $d\in \mathbb{N}^I$, denote by $S^{\dd}_w(\delta)$ be the subset of $S^d_w(\delta)$ with partitions $A=(d_i,w_i)_{i=1}^k$ for weights $w_i\in\mathbb{Z}$.




\subsubsection{}\label{ma}
For $A=(d_i,w_i)_{i=1}^k$ in $S^d_w(\delta)$, let $\mathbb{N}_A(\delta):=\otimes_{i=1}^k\mathbb{N}(d_i)_{w_i}$ and $\mathbb{M}_A(\delta):=\text{MF}^{\text{gr}}\left(\mathbb{N}_A(\delta),W\right)$, the category of matrix factorizations with factors in $\mathbb{N}_A(\delta)$. 

Define similarly $\mathbb{M}_A(\delta)$ for $A$ in $S^d_w(\delta)$. Let $\lambda_A:=\lambda_{d_1,\cdots,d_k}$ and $\chi_A:=(w_i)_{i=1}^k$ be the cocharacter of $SG(d)$ and the weight of $\left(\mathbb{C}^*\right)^k$ associated to $A$.

\subsubsection{}\label{boundedqp}
Let $a\in\mathbb{Z}$. Fix an ordered partition $(d_1,\cdots, d_k)$ of $d$ and let $\lambda:=\lambda_{d_1,\cdots, d_k}$ be an associated cocharacter of this partition. The set of $A=(d_i, w_i)_{i=1}^k$ in $S^d_w(\delta)$ such that $\langle \lambda, \F\rangle\geq a$ for $\F\in\mathbb{N}_A(\delta)$ is finite. The statement follows by induction on $\sum_{i=1}^k d_i$ and the fact that $\langle\lambda,\F\rangle$ is bounded above for any $A$ in $S^d_w(\delta)$.

\subsubsection{}\label{qpqp}
We now state \cite[Theorem 1.1]{P}.

\begin{thm}\label{qp}
Let $\delta\in M_{\mathbb{R}}^{\mathfrak{S}_d}$.
Consider $A=(d_i,w_i)_{i=1}^k$ in $S_w^d(\delta)$. Let $\lambda:=\lambda_{d_1,\cdots,d_k}$ be the cocharacter associated to $A$.
The functor
\[p_{\lambda*}q_\lambda^*: 
\mathbb{M}_{A}(\delta)\to \text{MF}^{\text{gr}}(\X(d),W)_w\] 
is fully faithful.
There is a semi-orthogonal decomposition
\[\text{MF}^{\text{gr}}(\X(d),W)_w=\Big\langle \mathbb{M}_{A}(\delta)\Big\rangle,\]
where the right hand side contains all $A$ in $S_w^d(\delta)$. The order of categories in the semi-orthogonal decomposition respects the order from Subsection \ref{weight}, see also Subsection \ref{sod4}.
\end{thm}


\subsubsection{}\label{l}
We explain how to construct the adjoint of $p_{\lambda*}q_\lambda^*$ from Theorem \ref{qp}. 

The torus $T=(\mathbb{C}^*)^k$ acts trivially on $\times_{i=1}^k\X(d_i)$, where the $i$th factor of $T$ is the group from Subsection \ref{weight}.
For an ordered tuplet $A=(d_i,w_i)_{i=1}^k$ with $k\geq 2$ in $S_w^d(\delta)$, let $\lambda=\lambda_{d_1,\cdots,d_k}$ and $\chi_A:=(w_i)_{i=1}^k$ be the cocharacter of $SG(d)$ and the weight of $T$ associated to $A$. Let \[\text{MF}^{\text{gr}}(\X(d),W)_{\text{below}}\subset \text{MF}^{\text{gr}}_{\text{qcoh}}(\X(d),W)\] be the subcategory of objects $\F$ such that $\beta_w(\F)$ is coherent for all weights $w\in\mathbb{Z}$ of $\lambda$ and it is zero for $w$ small enough. Define similarly $\text{MF}^{\text{gr}}(\X(d),W)_{\text{above}}$.
The functor
\[F:=p_{\lambda*}q_\lambda^*: \text{MF}^{\text{gr}}
\left(\times_{i=1}^k\X(d_i),W\right)_{\chi_A}\to \text{MF}^{\text{gr}}(\X(d),W)_{w}\]
is fully faithful and has a right adjoint
\begin{equation}\label{adjoint}
R:=q_{\lambda*}p_\lambda^!:\text{MF}^{\text{gr}}(\X(d),W)_{w}\to 
\text{MF}^{\text{gr}}
\left(\times_{i=1}^k\X(d_i),W\right)_{\text{above}}.
\end{equation}
The proof that the functor is fully faithful is as in \cite[Proposition 7.16]{P}. Its image lies in $\text{MF}^{\text{gr}}_{\text{above}}$ by \cite[Lemma 2.2.3]{T} for the cocharacter $\lambda^{-1}$.
We check that 
\begin{equation}\label{adjointt}
L:=\mathbb{D}_{\X^\lambda}R\mathbb{D}_\X:\text{MF}^{\text{gr}}(\X(d),W)_{w}\to 
\text{MF}^{\text{gr}}
\left(\times_{i=1}^k\X(d_i),W\right)_{\text{below}}
\end{equation}
is a left adjoint to $F$.
First, we have that 
\begin{equation}\label{df}
    \mathbb{D}_\X F=F\mathbb{D}_{\X^\lambda}.
\end{equation}
Indeed, for $\mathcal{A}$ in $\text{MF}^{\text{gr}}\left(\times_{i=1}^k\X(d_i),W\right)$, we have that
\[R\mathcal{H}om_\X\left(p_*q^*\mathcal{A},\omega_\X\right)=p_!q^!R\mathcal{H}om_{\X^\lambda}\left(\mathcal{A},\omega_{\X^\lambda}\right).\]
Let $\omega_p=\det \mathbb{L}_p[\dim p]$, $\omega_q=\det \mathbb{L}_q[\dim q]$, and let $\mathcal{C}$ be in $\text{MF}^{\text{gr}}\left(\times_{i=1}^k\X(d_i),W\right)$.
By the formulas in \cite[Subsection 2.2]{T} and using that $Q$ is symmetric, \[p_!q^!\mathcal{C}=p_*q^*(\mathcal{C}\otimes\omega_p\otimes\omega_q)=p_*q^*\mathcal{C}.\]
Next, let $\mathcal{A}$ be as above and let $\mathcal{B}$ be in $\text{MF}^{\text{gr}}(\X(d),W)$.
The statement \eqref{adjointt} to be proved is: \[R\text{Hom}_\X(\B,F\A)=R\text{Hom}_{\X^\lambda}\left(\mathbb{D}_\X R \mathbb{D}_{\X^\lambda}\B, \A\right).\] This follows from the natural isomorphisms:
\begin{multline*}
    R\text{Hom}_{\X^\lambda}(\mathbb{D}_\X R \mathbb{D}_{\X^\lambda}\B, \A)\cong
    R\text{Hom}_{\X^\lambda}(\mathbb{D}_{\X^\lambda}\A, R \mathbb{D}_{\X^\lambda}\B)\cong\\
    R\text{Hom}_\X(F\mathbb{D}_{\X^\lambda}\A, \mathbb{D}_\X \B)\cong
    R\text{Hom}_\X(\mathbb{D}_{\X}F\A, \mathbb{D}_\X \B)\cong 
    R\text{Hom}_\X(\B, F\A),
\end{multline*}
where the first and fourth isomorphisms are induced by duality, the second by adjunction of $F$ and $R$, and the third by \eqref{df}.
The functor $F$ has thus left adjoints
\begin{equation}\label{ladjoint}
\Delta_A:=\beta_{\chi_A} L:\text{MF}^{\text{gr}}(\X(d),W)_{w}\to 
\text{MF}^{\text{gr}}
\left(\times_{i=1}^k\X(d_i),W\right)_{\chi_A}.
\end{equation}
The functor $L$ in \eqref{adjointt} has image in $\text{MF}^{\text{gr}}_{\text{below}}$ and so $\Delta_A(\F)=0$ for all but finitely many sets $A$ in $S_w^d(\delta)$ by the discussion in Subsection \ref{boundedqp}.


\subsubsection{}\label{lrm}
By Theorem \ref{qp}, the inclusion $\mathbb{M}(d; \delta)\subset \text{MF}^{\text{gr}}(\X(d),W)$ has a right adjoint $R$. Then $L:=\mathbb{D}_\X R\mathbb{D}_\X$ is a left adjoint to the inclusion. 

\section{Comparison of Hall algebras via dimensional reduction}
\label{comparison}

\subsection{Koszul equivalence}\label{Koszuls}
\subsubsection{}\label{koszuls2}
Consider a quiver $\widetilde{Q}=\left(I, \widetilde{E}\right)$ where  $\widetilde{E}=E\sqcup C$. 
Let $Q=(I,E)$ and $Q'=(I,C)$.
The group $\mathbb{C}^*$ acts on representations of $\widetilde{Q}$ by scaling the linear maps corresponding to edges in $C$ with weight $2$. 
Consider a potential $\widetilde{W}$ of $\widetilde{Q}$ on which 
$\mathbb{C}^*$ acts with weight $2$. The set $C$ is called a \textit{cut} for $\left(\widetilde{Q}, \widetilde{W}\right)$ in the literature. 
Denote by $\widetilde{\X(d)}$ the moduli stack of representations of dimension $d$ for the quiver $\widetilde{Q}$ and by $\X(d)$ the analogous stack for the quiver $Q$. We consider the category of graded matrix factorizations $\text{MF}^{\text{gr}}\left(\widetilde{\X(d)}, W\right)$ with respect to the action of the group $\mathbb{C}^*$ mentioned above.
Denote the representation space of $Q'$ by 
$C(d)$, so $C(d)$ is the vector space
\[C(d):=\prod_{e\in C}\text{Hom}\left(\mathbb{C}^{d_{s(e)}}, \mathbb{C}^{d_{t(e)}}\right),\] where $s, t:C\to I$ are the source and target maps.
The space $C(d)$ has a natural action of $G(d):=\prod_{i\in I}GL(d_i)$ and thus there is a natural $G(d)$-equivariant vector bundle $\mathcal{O}_{R(d)}\otimes C(d)$. We abuse notation and also denote by $C(d)$ the corresponding vector bundle
on the stack $\X(d)$. 
Write \[\widetilde{W}=\sum_{c\in C}cW_c,\] where $W_c$ is a path of $Q$. Define the algebra $P=\mathbb{C}[Q]/I$, where $I$ is the two-sided ideal generated by $W_c$ for $c\in C$.
The potential $\widetilde{W}$ induces a section of the dual vector bundle $C(d)^{\vee}$, and thus a map $\partial:C(d)\to \mathcal{O}_{\X(d)}$.
The moduli stack of representations of $P$ of dimension $d$ is the Koszul stack
\begin{equation}\label{der}
\PP(d):=\text{Spec}\left(\mathcal{O}_{R(d)}\left[C(d)[1];\partial\right]\right)\big/ G(d).
\end{equation}
There is an equivalence of categories \cite{I}, \cite[Theorem 2.3.3]{T}, called the Koszul equivalence or dimensional reduction:
\begin{equation}\label{Koszul}
\Phi: D^b\left(\PP(d)\right)\cong \text{MF}^{\text{gr}}\left(\widetilde{\X(d)}, \widetilde{W}\right).
\end{equation}

\subsubsection{}\label{quivertilde}
Assume we are in the setting on Subsection \ref{Koszuls} and that $\widetilde{Q}$ is symmetric. 
Let $d\in\mathbb{N}^I$.
Recall that the representation space of $Q'$ is 
$C(d)$. We also denote by $C(d)$ the natural vector bundle on $\X(d)$. The dual $C(d)^{\vee}$ is naturally isomorphic to $\overline{C}(d)$, the representation space of the opposite quiver $\overline{Q'}$. 
There is a vector bundle $C(d)^{\lambda\leq 0}\cong \left(\overline{C}(d)^{\lambda\geq 0}\right)^{\vee}$ on $\X(d)^{\lambda\geq 0}$. 
For a cocharacter $\lambda$ of $SG(d)$, define
\[\PP(d)^{\lambda\geq 0}:=\text{Spec}\left(\mathcal{O}_{R(d)^{\lambda\geq 0}}\left[C(d)^{\lambda\leq 0}[1];\partial\right]\right)\big/ G(d).\]
For dimension vectors $d$ and $e$, let $\lambda:=\lambda_{d,e}$ be the cocharacter corresponding to the partition $(d,e)$ of $d+e$.
Define
$\PP(d,e):=\PP(d+e)^{\lambda\geq 0}.$
There are quasi-smooth maps $q_{d,e}$ and proper maps $p_{d,e}$, see \cite[Section 2.2]{VV} for the case of preprojective algebras:
\begin{equation}\label{e2}
   \PP(d)\times\PP(e) \xleftarrow{q_{d,e}}\PP(d,e)\xrightarrow{p_{d,e}}\PP(d+e).
\end{equation}
Recall the Koszul equivalence \eqref{Koszul}:
\begin{equation*}
\Phi: D^b\left(\PP(d)\right)\cong \text{MF}^{\text{gr}}\left(\widetilde{\X(d)}, \widetilde{W}\right).
\end{equation*}
We drop the subscripts for the maps $p$ and $q$ when they are clear. Let $\lambda$ be a cocharacter of $SG(d)$.
Define the line bundle on $\widetilde{\X(d)}^\lambda$:
\begin{equation}\label{omega}
\omega_\lambda:=\det \Big(C(d)^{\lambda\leq 0}\big/C(d)^\lambda\Big)
\end{equation}

\begin{prop}\label{comp}
Let $d$ be a dimension vector and $\lambda$ a cocharacter of $SG(d)$.
Define the functor
\[\widetilde{m'}:=\widetilde{p}_*\widetilde{q}^*(-\otimes\omega_\lambda): \mathrm{MF}^{\text{gr}}\left(\widetilde{\X(d)}^\lambda, \widetilde{W}\right)
\to 
\mathrm{MF}^{\text{gr}}\left(\widetilde{\X(d)}, \widetilde{W}\right).\]
The following diagram commutes:
\begin{equation*}
    \begin{tikzcd}
    D^b\left(\PP(d)^\lambda\right)
    \arrow[d, "\Phi"]\arrow[r,"p_*q^*"]& D^b\left(\PP(d)\right)\arrow[d,"\Phi"]\\
    \mathrm{MF}^{\text{gr}}\left(\widetilde{\X(d)}^\lambda, \widetilde{W}\right)
    \arrow[r,"\widetilde{m'}"]&
    \mathrm{MF}^{\text{gr}}\left(\widetilde{\X(d)}, \widetilde{W}\right).
    \end{tikzcd}
\end{equation*}
\end{prop}

\begin{proof}
Define the stacks
\begin{align*}
    \mathcal{W}&=\left(R(d)^{\lambda\geq 0}\oplus C(d)^{\lambda}\right)/G^{\lambda\geq 0}\\
    \mathcal{Y}&=\left(R(d)^{\lambda\geq 0}\oplus C(d)^{\lambda\leq 0}\right)\big/G^{\lambda\geq 0}\\
    \mathcal{Z}&=\left(R(d)^{\lambda\geq 0}\oplus C(d)\right)/G^{\lambda\geq 0}.
\end{align*}
There are natural maps, see for example the settings of \cite[Lemma 2.4.4, Lemma 2.4.7]{T}:
\begin{align*}
    q_1&:\mathcal{W}\to\widetilde{\X(d)}^\lambda\\
    q_2&:\mathcal{W}\to \mathcal{Y}\\
    p_1&:\mathcal{Z}\to \mathcal{Y}\\
    p_2&:\mathcal{Z}\to \widetilde{\X(d)}.
\end{align*}
The map $q$ is quasi-smooth, so the following diagram commutes by \cite[Lemma 2.4.7]{T}:
\begin{equation*}
    \begin{tikzcd}
    D^b\left(\PP(d)^\lambda\right)
    \arrow[r,"\Phi"]\arrow[dd,"q^*"]& 
    \text{MF}^{\text{gr}}\left(\widetilde{\X(d)}^\lambda, \widetilde{W}\right)\arrow[d,"q_1^*"]\\
    & \text{MF}^{\text{gr}}\left(\mathcal{W}, \widetilde{W}\right)\arrow[d,"q_{2!}"]\\
    D^b\left(\PP(d)^{\lambda\geq 0}\right)\arrow[r,"\Phi"]& \text{MF}^{\text{gr}}\left(\mathcal{Y}, \widetilde{W}\right)
    \end{tikzcd}
\end{equation*}
The map $p$ is proper, so the following diagram commutes by \cite[Lemma 2.4.4]{T}:
\begin{equation*}
    \begin{tikzcd}
    D^b\left(\PP(d)^{\lambda\geq 0}\right)\arrow[r,"\Phi"]\arrow[dd,"p_*"]& \text{MF}^{\text{gr}}\left(\mathcal{Y}, \widetilde{W}\right)\arrow[d,"p_1^*"]\\
    & \text{MF}^{\text{gr}}\left(\mathcal{Z}, \widetilde{W}\right)\arrow[d,"p_{2*}"]\\
    D^b(\PP(d))\arrow[r,"\Phi"] & \text{MF}^{\text{gr}}\left(\widetilde{\X(d)}, \widetilde{W}\right).
    \end{tikzcd}
\end{equation*}
There are natural maps $s$, $t$ in the following cartesian diagram:
\begin{equation*}
    \begin{tikzcd}
    \widetilde{\X(d)}^{\lambda\geq 0}\arrow[d,"s"]\arrow[r,"t"]&
    \mathcal{Z}\arrow[d,"p_1"]\\
    \mathcal{W}\arrow[r,"q_2"]&
    \mathcal{Y}.
    \end{tikzcd}
\end{equation*}
By proper base change, $p_1^*q_{2!}=t_{!}s^*$. Further, $s^*q_1^*\omega_\lambda=s^*\det \mathbb{L}_{q_2}=\det \mathbb{L}_{t}$ and so $t_{!}(F)=t_{*}(\mathcal{F}\otimes
s^*q_1^*\omega_\lambda)$ by \cite[Subsection 2.2.2]{T}. 
Let $\mathcal{E}$ be a complex in $D^b(\PP(d)^\lambda)$. We compute 
\[p_*q^*(\mathcal{E})=p_{2*}p_1^*q_{2!}q_1^*(\mathcal{E})=p_{2*}t_{!}s^*q_1^*(\mathcal{E})=\widetilde{p}_*\left(\widetilde{q}^*(\mathcal{E}\otimes\omega_\lambda)\right).\]
\end{proof}

\subsection{Semi-orthogonal decompositions for preprojective-like HAs}
\subsubsection{}\label{comp3}
Let $\lambda$ be a cocharacter of $SG(d)$
and recall the definition of $\omega_\lambda$ from \eqref{omega}. We claim that:
\begin{equation}\label{three}
    2\langle \lambda, \omega_\lambda\rangle=
    \big\langle 
    \lambda, \mathbb{L}_{\widetilde{\X(d)}}^{\lambda\leq 0}
    \big|_{\widetilde{\X(d)}^\lambda}
    \big\rangle-
    \big\langle \lambda, \mathbb{L}_{\PP(d)}^{\lambda\leq 0}\big|_{\PP(d)^\lambda}\big\rangle.
\end{equation}
This follows from the following equalities in $K_0\left(BG(d)\right)$:
\begin{align*}
    \left[\mathbb{L}_{\widetilde{\X(d)}}\right]&=[R(d)]+[C(d)]-[\mathfrak{g}(d)]\\
    \left[\mathbb{L}_{\PP(d)}\right]&=[R(d)]-[C(d)]-[\mathfrak{g}(d)]\\
    [\omega_\lambda]&=
    \left[C(d)^{\lambda\leq 0}\right].
\end{align*}

\subsubsection{}\label{four}
Consider $A=(d_i,w_i)_{i=1}^k$ an ordered partition of $(d,w)$ and with associated cocharacter $\lambda=\lambda_{d_1,\cdots,d_k}$. Let $A'$ be the ordered partition $(d_i,v_i)_{i=1}^k$, where $v_i$ are the corresponding weights of $\left(\boxtimes_{i=1}^k\mathcal{F}_i\right)\otimes\omega_\lambda$. 

Let $\delta\in M_{\mathbb{R}}^{\mathfrak{S}_d}$.
The set of all such $A'$ for $A\in S^d_w(\delta)$ is denoted by $V^d_w(\delta)$. 
The set $V^d_w(\delta)$ has an order induced from the order on $S^d_w(\delta)$.
The set of all such $A'$ for $A\in T^d_w(\delta)$ is denoted by $U^d_w(\delta)$.

\subsubsection{}\label{MT}
The category $\mathbb{M}(d; \delta)$ of $D^b\left(\PP(d)\right)$ is generated by complexes $\mathcal{F}$ such that for any cocharacter $\lambda$ of $SG(d)$, we have that
\[-\frac{n_\lambda}{2}+\langle \lambda, \delta\rangle\leq \big\langle\lambda, \iota_*\mathcal{F}|_0\big\rangle\leq \frac{n_\lambda}{2}+\langle \lambda, \delta\rangle,\]
where $\iota:\PP(d)\to \X(d)$ and $n_\lambda$ is the weight for the quiver $\widetilde{Q}$, see \eqref{weight2}. 

This category corresponds to the category $\mathbb{M}(d; \delta)$ of $\text{MF}^{\text{gr}}\left(\widetilde{\X(d)}, W\right)$ by \cite[Subsection 5.3, Lemma 5.3.8]{T}. 

\subsubsection{}\label{defM}
For $A=(d_i,w_i)_{i=1}^k$ in $V^d_w(\delta)$ or $U^d_w(\delta)$, 
let $\mathbb{M}_A(\delta):=\otimes_{i=1}^k\mathbb{M}(d_i)_{w_i}$ be a subcategory of $D^b\left(\times_{i=1}^k\PP(d_i)\right)_\chi$, where $\chi:=(w_i)_{i=1}^k$ is the weight corresponding to $A$. Using the Thom-Sebastiani theorem, the Koszul equivalence induces an equivalence:
\[\Phi:\mathbb{M}_A(\delta)\cong \mathbb{M}_{A^o}(\delta)\subset \text{MF}^{\text{gr}}\left(\times_{i=1}^k\widetilde{\X(d_i)},W\right)_\chi,\]
where $A^o\in S^d_w(\delta)$ or $T^d_w(\delta)$ with $(A^o)'=A$.

\subsubsection{}
By Theorem \ref{qp}, Proposition \ref{comp}, and the discussion in Subsection \ref{four}, we have that:

\begin{cor}\label{sod}
There is a semi-orthogonal decomposition
\[D^b(\PP(d))_w=\Big\langle \mathbb{M}_{A}(\delta)\Big\rangle,\]
where the right hand side contains all sets $A$ in $V_w^d(\delta)$. The order of the summands of the semi-orthogonal decomposition is as in Subsection \ref{four}, see also Subsection \ref{sod4}. 
\end{cor}

\subsubsection{}\label{ad}
We continue with the notation from the previous Subsection.
Consider an ordered tuple $A=(d_i,w_i)_{i=1}^k$ with $k\geq 2$ in $V_w^d(\delta)$. Let $\lambda=\lambda_{d_1,\cdots,d_k}$ and $\chi=(w_i)_{i=1}^k$ be the cocharacter of $SG(d)$ and the character of $T$ associated to $A$.
We explain the analogous results of Subsection \ref{l} in this setting. By the Koszul equivalence, the categories $D^b(\PP(d))$ 
have a dual functor $\mathbb{D}$. 
The fully faithful functor
\[p_{\lambda*}q_\lambda^*:  D^b\left(\times_{i=1}^k\PP(d_i)\right)_{\chi}\to D^b\left(\PP(d)\right)_w\]
has a right adjoint 
\begin{equation}\label{rightadjo}
R:=q_*p^!: D^b\left(\PP(d)\right)_w\to \text{Ind}\,D^b\left(\times_{i=1}^k\PP(d_i)\right)_{\text{above}},
\end{equation}
where the category $\text{Ind}\,D^b_{\text{above}}$ is defined as in Subsection \ref{l}. 
Indeed, using Proposition \ref{comp} and taking the adjoint of the two functors $\widetilde{m'}$ and $p_*q^*$, we have that
\[\Phi q_*p^!(\mathcal{E})=\left(\widetilde{q}_*\widetilde{p}^!\Phi(\mathcal{E})\right)\otimes\omega^{-1}_\lambda\] for $\mathcal{E}$ in $D^b\left(\PP(d)\right)$.
Thus the functor $q_*p^!$ has image in $\text{Ind}\,D^b_{\text{above}}$ by the discussion for $\text{MF}^{\text{gr}}$ in Subsection \ref{l}. 

Further, by the argument in Subsection \ref{l}, the functor $L=\mathbb{D}_{\PP^\lambda} R \mathbb{D}_{\PP}$ is a left adjoint of $p_{\lambda*}q_\lambda^*$ and has image in $\text{Ind}\,D^b_{\text{below}}$.   
Thus, for any weight $\chi$,
\[\beta_\chi L: D^b\left(\PP(d)\right)_w\to D^b\left(\times_{i=1}^k\PP(d_i)\right)_{\chi}.\]

The inclusion of $\mathbb{M}_A(\delta)$
in $D^b\left(\PP(d)\right)_w$ has a left adjoint by the argument in Subsection \ref{lrm}. Thus 
the fully faithful functor  
\[p_{\lambda*}q_\lambda^*: \mathbb{M}_A(\delta)\to D^b\left(\PP(d)\right)_w\]
has left adjoints 
\begin{equation}\label{adjoint2}
 \Delta_{A}: D^b\left(\PP(d)\right)_w\to \mathbb{M}_A(\delta).
\end{equation}

\subsubsection{}\label{boundedqp2}
Let $\F$ be in $D^b\left(\PP(d)\right)_w$.
By the discussion in Subsection \ref{boundedqp} and the fact that $L$ has image in $\text{Ind}\,D^b_{\text{below}}$, $\Delta_A(\F)=0$ for all but finitely many $A$ in $V^d_w(\delta)$.

\subsubsection{}

Let $J$ be the Jordan quiver, let $\widetilde{J}$ be the quiver with three loops $x, y, z$, and let $\widetilde{W}:=xyz-xzy$. The stacks $\mathfrak{P}(d)$ from Subsection \ref{koszuls2} recover the 
the moduli stacks of points on $\mathbb{A}^2_\mathbb{C}$ which we denote by $\mathfrak{C}_d$ for $d\in\mathbb{N}$. Denote its coarse space by $C_d$.

The categories $\mathbb{M}_A(\delta)$ do not depend on the weight $\delta\in M_{\mathbb{R}}^{\mathfrak{S}_d}$, and we will assume that $\delta=0$ and drop it from the notation. 
We want to describe in more detail the sets $V^d_w$ and $U^d_w$. Let $\chi$ be a dominant weight in $M(d)$ and write 
\[\chi+\rho=-\sum_{j\in T} r_jN_j+\psi\] as in \eqref{godown5}. Let $\lambda$ be the corresponding antidominant cocharacter with associated partition $(d_1,\cdots, d_k)$. Consider the weight 
\[\theta_\chi:=-\sum_{j\in T} r_jN_j-\rho^{\lambda<0}+\omega_\lambda.\]
Define $\mathfrak{g}_j$ analogously to $N_j$ for the adjoint representation of $GL(d)$ for $j\in T$. 
Then
\[\theta_A:=-\sum_{j\in T}\left(3r_j-\frac{3}{2}\right)\mathfrak{g}_j+c\tau_d=\sum_{i=1}^k w_i\tau_{d_i}\] is a dominant weight and so $\frac{w_1}{d_1}>\cdots >\frac{w_k}{d_k}$. Conversely, 
consider a partition $(d_i, w_i)_{i=1}^k$ with $\frac{w_1}{d_1}>\cdots >\frac{w_k}{d_k}$ and let $\psi_A:=\sum_{i=1}^k w_i\tau_{d_i}$. Then $\psi_A$ is a dominant weight and $\psi_A$ is a linear combination of $\mathfrak{g}_j$ for a tree $T$: 
\begin{equation}\label{chiAA}
    \psi_A=\sum_{i=1}^s w_i\tau_{d_i}=-\sum_{j\in T}\left(3r_j-\frac{3}{2}\right)\mathfrak{g}_j+c\tau_d.
\end{equation}
Then $r_i>r_j$ for $i, j\in T$ such that there exists a path from $i$ to $j$, otherwise $\psi_A$ is not dominant.

Thus the set $V^d_w$ contains partitions $(d_i, w_i)_{i=1}^k$ with $\frac{w_1}{d_1}>\cdots >\frac{w_k}{d_k}$. 
By the above discussion and Corollary \ref{sod}, we see that Theorem \ref{thm} holds for $\mathbb{A}^2_{\mathbb{C}}$: 

\begin{cor}\label{thmA}
There is a semi-orthogonal decomposition
\begin{equation}\label{sod222}
    D^b\left(\mathfrak{C}_d\right)_w=\Big\langle 
\mathbb{M}_A\Big\rangle,\end{equation}
where the right hand side contains all partitions $A\in V^d_w$.
The semi-orthogonal decomposition holds over $C_d$ in the following sense. Let $A<B$ be two partitions and let $\mathcal{F}\in \mathbb{M}_A$ and $\mathcal{G}\in \mathbb{M}_B$. Then \begin{equation}\label{zeroo}
    R\pi_{d*}\left(R\mathcal{H}om_{\mathfrak{C}_d}(\mathcal{F}, \mathcal{G})\right)=0.\end{equation}
\end{cor}

Note that the statement \eqref{zeroo} holds from the semi-orthogonality property because $C_d$ is affine.
Similarly, the set $U^d_w$ contains partitions $A=(d_i,w_i)_{i=1}^k$ such that $\psi_A$ from \eqref{chiAA} is a multiple of $\tau_d$, which implies that
\begin{equation}\label{equ}
    \frac{w_1}{d_1}=\cdots=\frac{w_k}{d_k}.\end{equation}
Conversely, all the partitions $A$ satisfying \eqref{equ} are in $U^d_w$.

\section{Hall algebras of surfaces}\label{5}

\subsection{Moduli of sheaves on surfaces via quivers}

\subsubsection{}\label{bbb}
Let $S$ be a smooth surface and let $d\in\mathbb{N}$.
Consider the coarse moduli space morphism \[\pi_d:\mathfrak{M}_{d}
\to M_d:=\text{Sym}^d(S).\]
Choose $k$ distinct points $p_1,\cdots, p_k$ with multiplicities $d_1,\cdots,d_k$ and let $p\in \text{Sym}^d(S)$ be the corresponding point.
Denote by $\widehat{\mathfrak{M}_{d, p}}$ the formal completion of $\mathfrak{M}_d$ along $\pi_d^{-1}(p)$. Note that 
\begin{equation}\label{isoHalP}
\widehat{\mathfrak{M}_{d, p}}\cong \mathfrak{M}_{d}\times_{M_d} \widehat{M_{d, p}},\end{equation} where $\widehat{M_{d, p}}$ is the completion of $M_d$ at $p$, see \cite[Section 2.1]{HLPr}.

Denote by $J'$ the doubled quiver of $J$. Consider the quiver $Q=\bigsqcup_{i=1}^k J'$ with vertex set $I=\{1,\cdots,k\}$ and consider the dimension vector $d=(d_i)_{i=1}^k\in\mathbb{N}^k$. The quiver $Q$ is the Ext quiver of the polystable sheaf $\bigoplus_{i=1}^k \mathcal{O}_{p_i}^{\oplus d_i}$.
Consider the (Koszul) stack
\begin{equation}\label{PP}
\widehat{\mathfrak{P}(\dd)}:=\times_{i=1}^k \widehat{\mathfrak{C}(d_i)},
\end{equation}
where the completions for the spaces on the right hand side are at zero.
Let $\widehat{P(\dd)}$ be the coarse space of $\widehat{\PP(\dd)}^{\text{cl}}$. The corresponding coarse space map is $\pi_{\dd}:=\times_{i=1}^k \pi_{d_i}$.
By \cite[Lemma 5.4.1, Section 7.4]{T}, we have canonical isomorphisms $\nu$ which extend to isomorphisms over analytic neighborhoods of $p\in M_d$ and $0\in P(\dd)$:
\begin{equation}\label{dd1}
\begin{tikzcd}
\mathfrak{M}_{d} \arrow[d, "\pi_d"] & \widehat{\mathfrak{M}_{d, p}} \arrow[l] \arrow[d, "\pi_d"] \arrow[r, "\sim", "\nu"']& \widehat{\PP(\dd)}\arrow[d, "\pi_{\dd}"] \\
M_d& \widehat{M_{d, p}} \arrow[l] \arrow[r, "\sim"]& \widehat{P(\dd)}.
\end{tikzcd}
\end{equation}


\subsubsection{}\label{oo} For $\underline{d}=(d_1,\cdots, d_k)$ a partition of $d$, let $M_{\underline{d}}\subset M_d$ be the locus of points $p\in M_d$ corresponding to (not necessarily distinct) points $p_1,\cdots, p_k$ with multiplicities $d_1,\cdots, d_k$.  
By a d\'{e}vissage argument, the category $D^b\left(\mathfrak{M}_d\right)$ is generated by locally free sheaves on $\pi_d^{-1}\left(M_{\underline{d}}\right)$ for all partitions $\underline{d}$ of $d$. 

For $U\subset M_d$ an open analytic subset, denote by $D^b_o\left(\pi_d^{-1}(U)\right)$ the category of coherent (analytic) sheaves generated by restrictions of coherent sheaves on $D^b\left(\mathfrak{M}_d\right)$. For $U$ a small contractible open subset of $M_d$, the category $D^b_o\left(\pi_d^{-1}(U)\right)$ is generated by sheaves $\mathcal{O}_
{\pi_n^{-1}\left(U\cap M_{\underline{d}}\right)}\otimes \Gamma(\chi)$ for $\chi$ a weight of $GL(d)$ and $\underline{d}$ a partition of $d$.

\subsubsection{}\label{ad2} Assume $S=\mathbb{A}^2_{\mathbb{C}}$ and let $p\in M_d$ as above. We obtain a commutative diagram:
\begin{equation}
    \begin{tikzcd}\label{dd2}
    \mathfrak{P}(\dd):=\times_{i=1}^k \mathfrak{C}(d_i)\arrow[d, "\pi_{\dd}"]\arrow[r]& \mathfrak{C}(d)\arrow[d, "\pi_d"]\\
    P(\dd):=\times_{i=1}^k C(d_i)\arrow[r]& C(d).
    \end{tikzcd}
\end{equation}
The thick subcategory of $D^b\left(\mathfrak{P}(\dd)\right)$ generated by restrictions of coherent sheaves on $\mathfrak{C}(d)$ is $D^b\left(\mathfrak{P}(\dd)\right)$. 

For $p\in M_d$, define $D^b_o\left(
\widehat{\mathfrak{M}_{d,p}}\right)$ as the thick subcategory of $D^b\left(
\widehat{\mathfrak{M}_{d,p}}
\right)$ generated by restrictions of sheaves $\mathcal{O}_
{\pi_d^{-1}(M_{\underline{d}})}\otimes \Gamma(\chi)$ for $\chi$ a dominant weight of $GL(d)$ and $\underline{d}$ a partition of $d$. 
The category $D^b_o\left(
\widehat{\mathfrak{M}_{d,p}}\right)$ depends only on the Ext quiver of $p$. Indeed, this follows from the isomorphisms in \eqref{dd1} and \eqref{dd2} for analytic open neighborhoods of the point $0$ in $P(\dd)$, the point $p$ in $C(d)$, and $p$ in $M_d$. 

The semi-orthogonal decomposition \eqref{sod222} induces a semi-orthogonal decomposition: \begin{equation}\label{sod333}
    D^b_o\left(\widehat{\mathfrak{P}(\dd)}\right)=\Big\langle \widehat{\mathbb{M}_{A,p}}\Big\rangle,
    \end{equation}
where we denote by $\widehat{\mathbb{M}_{A,p}}$ the subcategory of $D^b_o\left( \widehat{\mathfrak{P}(\dd)}\right)$ generated by restrictions of sheaves in $\mathbb{M}_A$. Further, the discussion in Subsection \ref{ad} applies to the categories from \eqref{sod333} and provides adjoint functors $\Delta_A$ to the multiplication functors for $A\in V^d_w$.

\subsubsection{} 
The Hall products for the surface $S$ and the preprojective algebra of the quiver $Q$ are compatible. Let $b, c\in \mathbb{N}$, let $a=b+c$, and let $p\in M_b$, $q\in M_c$, and $r:=p\oplus q\in M_a$. 
Let $Q$ be the Ext-quiver (for $S$) of $a$ and let $d$ and $e$ be the dimension vector of $Q$ corresponding to $b$ and $c$.
Denote by $\widehat{\PP}$ the derived stack defined in \eqref{PP} for $a$. 
Then there is a commutative diagram
\begin{equation*}
\begin{tikzcd}
\mathfrak{M}_b
\times
\mathfrak{M}_c & \widehat{\mathfrak{M}_{b, p}}\times
\widehat{\mathfrak{M}_{c,q}} 
\arrow[l] \arrow[r, "\sim", "\nu\times\nu"'] & 
\widehat{\PP(d)}\times \widehat{\PP(e)}\\
\mathfrak{M}_{b,c} \arrow[u, "q_{b,c}"] \arrow[d, "p_{b,c}"'] & \widehat{\mathfrak{M}_{b,c, r}}
\arrow[u, "q_{b,c}"] \arrow[d, "p_{b,c}"'] \arrow[l] \arrow[r, "\sim", "\nu"']& \widehat{\PP(d,e)} \arrow[u, "q_{d,e}"] \arrow[d, "p_{d,e}"'] \\
\mathfrak{M}_a & \widehat{\mathfrak{M}_{a,r}} \arrow[l] \arrow[r, "\sim", "\nu"']& \widehat{\PP(d+e)}.
\end{tikzcd}
\end{equation*}

\subsubsection{}
For any coherent sheaf $\F$ on $S$, there is an inclusion $\mathbb{C}^*\hookrightarrow \text{Aut}\left(\F\right)$
given by scaling. 
Thus the inertia stack of $\MM_d$ contains a natural copy of $\mathbb{C}^*$. Denote by $D^b\left(\MM_d\right)_w$ the category of complexes on $\MM_d$ on which $\mathbb{C}^*$ acts with weight $w$.

\subsubsection{}\label{M}
For $p\in M_d$, consider the stack $\widehat{\PP(\dd)}$ defined in \eqref{PP}. Recall the natural isomorphism \[\nu: \widehat{\mathfrak{M}_{d, p}}\cong \widehat{\PP(\dd)}.\]
There is thus an equivalence of categories:
\[\nu^{-1}: D^b\left(\widehat{\PP(\dd)}\right)\cong D^b\left(\widehat{\mathfrak{M}_{d,p}}\right).\]
We have that $\widehat{\mathbb{M}(d)_p}\cong\nu^{-1}\left(\widehat{\mathbb{M}(d)}\right)$.
Consider the restriction functor: \[\Psi_p: D^b\left(\mathfrak{M}_d\right)\to D^b\left(\widehat{\mathfrak{M}_{d,p}}\right).\] 
Define $\mathbb{M}(d)$ as the subcategory of $D^b\left(\mathfrak{M}_d\right)$ generated by complexes $\F$ such that for any point $p$ in $M_d$, we have that
\[\Psi_p(\F)\in \widehat{\mathbb{M}(d)_p}\subset D^b\left(\widehat{\mathfrak{M}_{d,p}}\right).\]
For $A=(d_i,w_i)_{i=1}^k\in V^d_w$ or $U^d_w$, define $\mathbb{M}_A:=\otimes_{i=1}^k \mathbb{M}(d_i)_{w_i}$.



\subsection{Proof of Theorem \ref{thm}}
We use induction on $d$. The stacks $\MM_d$ have dualizing functors $\mathbb{D}$ \cite[Subsection 2.2.1]{HL}. If $d=1$, then $D^b\left(\mathfrak{M}_1\right)_w\cong D^b(S)\cong \mathbb{M}(1)_w$ for any $w\in\mathbb{Z}$.

Let $\F$ be a complex in $D^b\left(\MM_d\right)$, let $A=(d_i, w_i)_{i=1}^k$ in $V^d_w$ with $k\geq 2$, and let $\lambda$ be its associated cocharacter.
We have that 
\[\mathbb{D} q_{\lambda*}p_\lambda^!\mathbb{D}\F\in \text{Ind}\,D^b\left(\times_{i=1}^k\MM_{d_i}\right)_{\text{below}}.\] 
Indeed, from the local statement \eqref{rightadjo} and \eqref{dd1}, the statement is true over a small neighborhood $U$ of $p\in M_d$; by embedding $S$ in a projective surface, we see that $M_d$ can be covered by a finite number of such open sets $U$. 
Thus, from Subsection \ref{boundedqp2}, we have that $\Delta_A(\F)=0$ for all but finitely $A$ in $V^d_w$. 



Consider $A=(d_i,w_i)_{i=1}^k$ and $B=(e_i,v_i)_{i=1}^s$ in $V^d_w$ such that $(d,w)<A<B$. 
Let $\lambda$ and $\mu$ be the associated cocharacters for $A$ and $B$ and let $\chi=(w_i)_{i=1}^k$ be the weight corresponding to $A$. 
Using the induction hypothesis and the argument in Subsection \ref{lrm}, the inclusion of $\mathbb{M}_A$ in $D^b\left(\times_{i=1}^k\mathfrak{M}_{d_i}\right)_\chi$ has a left adjoint
\[\Phi_A: D^b\left(\times_{i=1}^k\mathfrak{M}_{d_i}\right)_\chi\to \mathbb{M}_A.\]
The functor 
\[p_{\lambda*}q_\lambda^*: \mathbb{M}_A\to D^b\left(\MM_d\right)_w\]
is fully faithful and has a left adjoint 
\begin{equation}\label{adjoint3}
\Delta_A:= \Phi_{A}\beta_\chi \mathbb{D} q_{\lambda*}p_\lambda^!\mathbb{D}: D^b\left(\MM_d\right)_w\to \mathbb{M}_{A}.
\end{equation}
Both statements 
follow from the analogous statements for a formal completion $\widehat{\mathfrak{M}_{d,p}}$ for $p\in M_d$, see the discussion in the Subsections \ref{ad} and \ref{ad2}.


Next, let $\F_i\in \mathbb{M}(d_i)_{w_i}$ for $1\leq i\leq k$. We claim that: 
\[\Delta_B p_{\lambda*}q_\lambda^*\left(\boxtimes_{i=1}^k \F_i\right)=0.\]
Once again, this follows from the analogous statement for a formal completion $\widehat{\mathfrak{M}_{d,p}}$ for $p\in M_d$, see Subsections \ref{ad} and \ref{ad2}.

Let $\mathbb{W}$ be the left complement of the categories $\mathbb{M}_A$ for all sets $A=(d_i,w_i)_{i=1}^k$ in $V^d_w$ with $k\geq 2$. There is a semi-orthogonal decomposition 
\begin{equation}\label{ddss}
    D^b\left(\MM_d\right)_w=\Big\langle \mathbb{M}_A, \mathbb{W}\Big\rangle,
    \end{equation}
with summands corresponding to all $A$ in $V_w^d$ with $k\geq 2$.
For $p\in M_d$, define $\widehat{\mathbb{W}_p}$ as the subcategory of $D^b_o\left(\widehat{\MM_{d,p}}\right)_w$ generated by restrictions of sheaves in $\mathbb{W}$.
The restriction of the semi-orthogonal decomposition \eqref{ddss} implies an analogous semi-orthogonal decomposition 
\[D^b_o\left(\widehat{\MM_{d,p}}\right)_w=
\Big\langle \widehat{\mathbb{M}_A},  \widehat{\mathbb{W}_p}\Big\rangle\] 
for $p\in M_d$. Indeed, consider partitions $(d,w)<A<B$ and let $\mathcal{F}\in \mathbb{M}_A$ and $\mathcal{G}\in \mathbb{M}_B$. Then \[R\pi_{d*}\left(R\mathcal{H}om_{\mathfrak{M}_d}(\mathcal{F}, \mathcal{G})\right)=0\] from the semi-orthogonal decomposition \eqref{sod333} and Corollary \ref{thmA} applied for all points $p\in M_d$.
Next, the category $\mathbb{W}$ contains sheaves $\mathcal{F}$ such that $\Delta_A(\mathcal{F})=0$ for all partitions $A=(d_i,w_i)_{i=1}^k\in V^d_w$  with $k\geq 2$. The restriction of the functor $\Delta_A$ to $\widehat{\mathfrak{M}_{d,p}}$ is the analogously defined functor \[\Delta_A: \widehat{\mathfrak{M}_{d,p}}\to \widehat{\mathbb{M}_{A,p}}.\]
This construction further shows that $\widehat{\mathbb{W}_p}=\left(\widehat{\mathbb{M}(d)_p}\right)_w$ and thus that $\mathbb{W}=\mathbb{M}(d)_w$, and it also shows that \[R\pi_{d*}\left(R\mathcal{H}om_{\mathfrak{M}_d}(\mathcal{F}, \mathcal{G})\right)=0\]
for $\mathcal{F}\in \mathbb{M}(d)_w$ and $\mathcal{G}\in \mathbb{M}_A$.

\section{PBW theorem for surfaces}\label{6}

The proof of Theorem \ref{thm2} follows closely the proof of \cite[Proposition 5.5]{P}. We will explain how to modify the proof in loc. cit. to obtain the proof of Theorem \ref{thm2}. All the K-theoretic spaces in this Section are over the rational numbers $\mathbb{Q}$.

\subsection{The Shuffle algebra}
\subsubsection{}
Let $\text{Sh}(S)$ be the shuffle algebra of the surface $S$ considered by Negu\c{t} \cite{N}, Zhao \cite[Section 5]{Z}. Its underlying $\mathbb{N}$-graded vector space is \[\text{Sh}(S):=\bigoplus_{n\geq 1}\left(K_0(S^n)(z_1,\cdots, z_n)\right)^{\mathfrak{S}_n}.\]
Fix $n, m\geq 1$.
For $1\leq i, j\leq n+m$, let $\Delta_{ij}\subset S^{n+m}$ be the inclusion of the locus with equal points on the $i$th and $j$th copies of $S$. Define 
\[\zeta_{ij}(z):=1+\frac{z\cdot [ \mathcal{O}_{\Delta_{ij}}]}{(1-z)(1-z[\omega_S])}\in K_0\left(S^{n+m}\right)(z),\]
see \cite[Equation 3.6, Proposition 5.24]{N}.
Let $f$ and $g$ be elements of $\text{Sh}(S)$ of degrees $n$ and $m$, respectively.
The product on the shuffle algebra is defined by
\[(f*g)(z_1,\cdots,z_{n+m}):=\text{Sym}\left(f(z_1,\cdots, z_n)g(z_{n+1},\cdots,z_{n+m})\prod_{\substack{1\leq i\leq n\\ n+1\leq j\leq n+m}}\zeta_{ij}\left(\frac{z_i}{z_j}\right)\right),\] where the right hand side is the symmetrization after all cosets in $\mathfrak{S}_{n+m}/\mathfrak{S}_n\times\mathfrak{S}_m$.
There is an algebra morphism 
\[\Phi_S: \text{KHA}(S)\to \text{Sh}(S)\]  defined in \cite[Section 5.2]{Z}.
Let $S=\mathbb{A}^2_\mathbb{C}$ and consider the map \[\iota_d: \mathfrak{C}(d)\to \mathfrak{gl}(d)^2/GL(d).\] 
In \cite[Subsection 5.1.2]{P}, we constructed algebra morphisms by setting the potential to zero, so, for example, a morphism $\text{KHA}\left(\widetilde{J}, \widetilde{W}\right)\to \text{KHA}\left(\widetilde{J}, 0\right)$. Using dimensional reduction, the morphism is the same as the pushforward (up to an equivariant factor, see Proposition \ref{comp}):
\begin{equation}\label{iota}
    \iota_*: \bigoplus_{d\in\mathbb{N}} G_0\left(\mathfrak{C}(d)\right)\to 
\bigoplus_{d\in\mathbb{N}} K_0\left(\mathfrak{gl}(d)^2/GL(d)\right).
\end{equation}
Write $\mathfrak{C}(d)=C(d)/GL(d)$. 
Let $k_d: D(d)\to C(d)$ be the locus of diagonal matrices and let $\mathfrak{D}(d)=D(d)/T(d)$. Consider the map
\begin{multline}\label{k}
    K_0\left(\mathfrak{gl}(d)^2/GL(d)\right)\cong K_0\left(\mathfrak{gl}(d)^2/T(d)\right)^{\mathfrak{S}_d}\xrightarrow{k_d^*} K_0\left(\mathfrak{D}(d)\right)^{\mathfrak{S}_d}\to\\ \left(K_0(S^d)(z_1,\cdots, z_d)\right)^{\mathfrak{S}_d}
    .\end{multline}
    Write $k^*:=\bigoplus_{d\in \mathbb{N}}k_d^*$. Then $\Phi_{\mathbb{A}^2_\mathbb{C}}=k^*\iota_*$. 
    
\subsubsection{} There is a K-theoretic Hall algebra defined using topological K-theory by applying the Blanc topological K-theory functor $K^{\text{top}}$ \cite{Bl} to the Porta--Sala monoidal category \cite{PS}. For a quotient stack $\X=X/G$ with $X$ a possibly singular variety and $G$ a reductive group, $K^{\text{top}}_{\cdot}\left(\mathfrak{X}\right):=K^{\text{top}}_{\cdot}\left(\text{Perf}(\mathfrak{X})\right)$ is the Atiyah-Segal equivariant K-theory of $\mathfrak{X}$, while $G^{\text{top}}_{\cdot}\left(\mathfrak{X}\right):=K^{\text{top}}_{\cdot}\left(D^b\text{Coh}(\mathfrak{X})\right)$ is the 
Borel-Moore equivariant K-homology (also called Spanier-Whitehead) of $\mathfrak{X}$, see \cite[Theorem 3.9, Remark 0.1]{HLP}. 

We denote by $\text{KHA}^{\text{top}}(S):=\bigoplus_{d\in\mathbb{N}} G_0^{\text{top}}\left(\mathfrak{M}_d\right)$. 

\subsubsection{}\label{surfanew}
There is a topological Chern character \[\text{ch}^{\text{top}}: G_0^{\text{top}}\left(\mathfrak{M}_d\right)\to \prod_{i\in \mathbb{Z}} H^{\text{BM}}_{2i}(\mathfrak{M}_d, \mathbb{Q})\] factoring the usual Chern character $\text{ch}: G_0\left(\mathfrak{M}_d\right)\to \prod_{i\in \mathbb{Z}} H^{\text{BM}}_{2i}(\mathfrak{M}_d, \mathbb{Q})$. The map $\text{ch}^{\text{top}}$ is injective, see the explanations about the map \eqref{subsection5310}.

Davison \cite[Theorem A, Subsection 4.1.1]{Dav1}, \cite[Theorems B and C]{D} proved a BBDG Decompositon Theorem for the complex $R\pi_{d*}\omega_{\mathfrak{M}_d}$ which implies a PBW-type Theorem for the CoHA of points of $S$, see also \cite[Theorem 7.1.6]{KV}. The CoHA is generated by the cohomology of BPS sheaves, which are certain constructible sheaves on $M_d$ for every $d\geq 1$.
The BPS summand in dimension $d$ of $R\pi_{d*}\omega_{\mathfrak{M}_d}$ is $\text{IC}_S$, where $S\hookrightarrow M_d=\text{Sym}^d(S)$ is the small diagonal, see also the proof of Proposition \ref{pp1}.

\subsubsection{}\label{surfa}

Let $(U_j)_{j\in J}$ be a cover of $M_d$ with open subsets. Using the results from Subsection \ref{surfanew}, we obtain that the restriction maps induce an injection \[G_0^{\text{top}}\left(\mathfrak{M}_d\right)\hookrightarrow \bigoplus_{j\in J}G_0^{\text{top}}\left(\pi_d^{-1}(U_j)\right).\]

For a surface $S$, there is also a topological K-theoretic shuffle algebra $\text{SH}^{\text{top}}(S)$ and an algebra morphism $\Phi_S^{\text{top}}: \text{KHA}^{\text{top}}(S)\to \text{SH}^{\text{top}}(S)$ defined as above.


Let $U\subset S$ be a small open analytic subset of a point $p$ such that $\text{Sym}^d(U)\subset M_d$ is an open neighborhood of $(p,\cdots, p)\in M_d$ as in \eqref{dd1}.
One way to obtain $\Phi_S^{\text{top}}$ is to glue the analogous of the composition of the local maps \eqref{iota} and \eqref{k}:
\begin{equation}\label{kiota}
    k_d^*\iota_{d*}: G_i^{\text{top}}\left(\pi_d^{-1}\left(\text{Sym}^d(U)\right)\right)\to \left(K_i^{\text{top}}\left(U^d\right)(z_1,\cdots,z_d)\right)^{\mathfrak{S}_d}\end{equation} 
    for $i\in\mathbb{Z}$.
    To obtain a cover of $M_d$, we also need to use the corresponding maps for the open neighborhoods $\times_{i=1}^k \text{Sym}^{d_i}(U_i)\subset M_d$ associated to small neighborhoods $p_i\in U_i\subset S$.
We denote the image of $\Phi^{\text{top}}_S$ by $\text{KHA}'^{\text{top}}(S)$.

\subsection{A coproduct-type map}
Let $C>A$ be in $U^d_w$. The construction \eqref{adjoint3} also provides a functor
\[\Delta_{AB}: \mathbb{M}_A\to \mathbb{M}_C.\]
Consider pairs $(b,t)$, $(c,s)$, $(e,v)$, $(f,u)$, and $(d,w)$ in $\mathbb{N}\times\mathbb{Z}$ such that 
\[(e,v)+(f,u)=(b,t)+(c,s)=(d,w).\]
We denote by $A$ the two term partition $(e,v)$, $(f,u)$, and by $B$ the two term partition $(b,t)$, $(c,s)$. Assume that $A$ and $B$ are in $U^d_w$. 
Let $\textbf{S}$ be the set of partitions $C$ of $U^d_w$ with terms $(a_i, \alpha_i)$ for $1\leq i\leq 4$, some of them possibly zero, such that
\begin{align*}
    (a_1, \alpha_1)+(a_2, \alpha_2)&=(e,v),\\
    (a_3, \alpha_3)+(a_4, \alpha_4)&=(f,u),\\
    (a_1, \alpha_1)+(a_3, \alpha_3)&=(b,t),\\
    (a_2, \alpha_2)+(a_4, \alpha_4)&=(c,s),
\end{align*}
Define $\widetilde{m \boxtimes m}:=m\left(1\boxtimes \text{sw}\boxtimes 1\right).$ 

\begin{thm}\label{thm51}
The following diagram commutes:
\begin{equation*}
    \begin{tikzcd}
    K_0'^{\text{top}}(\mathbb{M}_A)
    \arrow[d, "\Delta_{AC}"]\arrow[r, "m"]& 
    K_0'^{\text{top}}\left(\mathbb{M}(d)_w\right)\arrow[d, "\Delta_B"]\\
    \bigoplus_{C\in \textbf{S}}
    K_0'^{\text{top}}(\mathbb{M}_C)
    \arrow[r, "\widetilde{m\boxtimes m}"]& K_0'^{\text{top}}(\mathbb{M}_B).
    \end{tikzcd}
\end{equation*}
\end{thm}

\begin{proof}
The construction of the maps $(\Phi_S)_d=k_d^*\iota_{d*}$ is local on $M_d$. For an analytic open subset $O\subset M_d$, let $\mathbb{M}_A(O)$ be the subcategory of $D^b_o\left(\pi_d^{-1}(O)\right)$, see Subsection \ref{oo}, generated by restrictions of sheaves on $\mathbb{M}_A\subset D^b\left(\mathfrak{M}_d\right)$.
By the discussion in Subsection \ref{surfa}, 
it suffices to show the analogous statement for $U\subset S$ an open analytic subset as in Subsection \ref{surfa}:
\begin{equation}
    \begin{tikzcd}
    K_0'^{\text{top}}(\mathbb{M}_A(O))
    \arrow[d, "\Delta_{AC}"]\arrow[r, "m"]& 
    K_0'^{\text{top}}\left(\mathbb{M}(d; O)_w\right)\arrow[d, "\Delta_B"]\\
    \bigoplus_{C\in \textbf{S}}
    K_0'^{\text{top}}(\mathbb{M}_C(O))
    \arrow[r, "\widetilde{m\boxtimes m}"]& K_0'^{\text{top}}(\mathbb{M}_B(O)),
    \end{tikzcd}
\end{equation}
where $O:=\text{Sym}^d(U)\subset M_d$.
We may assume that $U\subset \mathbb{A}^2_{\mathbb{C}}$. The argument follows as in the global case $\mathbb{A}^2_\mathbb{C}$ using explicit shuffle formulas for the maps involved in the quiver with zero potential $\left(\widetilde{J}, 0\right)$, see \cite[Theorem 5.3]{P}, \cite[Theorem 5.2]{P2}. 
Note that the formula for $\widetilde{m\boxtimes m}$ is more complicated in loc. cit., but it simplifies to the above description for the quiver $\widetilde{J}$ and by Proposition \ref{comp}. 
\end{proof}

\subsection{Primitive generators of the KHA} 
\subsubsection{}\label{subsection531}
We first discuss two preliminary results.

\begin{prop}\label{pp1}
Let $S$ be a surface with $H^1(S,\mathbb{Q})=0$. Then $H^{\text{BM}}_{\text{odd}}\left(\mathfrak{M}_d,\mathbb{Q}\right)=0$. 
\end{prop}

\begin{proof}
Recall the map 
\[\pi_d: \mathfrak{M}_d\to M_d.\]
We use the Decomposition Theorem of \cite[Theorem C, Subsection 1.2.3]{D} for the sheaf $R\pi_{d*}\left(\omega_{\mathfrak{M}_d}\right)$. Then the summands have even shifts, which follows by the direct computation for the stack of commuting matrices, or alternatively from the computation of the BPS sheaves of $\left(\widetilde{J}, \widetilde{W}\right)$ \cite[Theorem 5.1]{D2}. 
The summands are of the form $IC_{M_{\underline{d}}}(\mathcal{L})$ for certain local systems $\mathcal{L}$, and these summands appear also in the Decomposition Theorem for the complex $r_{\dd *} \left(\omega_{\mathfrak{M}_{\dd}}\right)$, where
\[r_{\dd}: \mathfrak{M}_{\dd}:=\times_{i=1}^k \mathfrak{M}_{d_i}\to \times_{i=1}^k M_{d_i}\to M_d,\] 
see for example \cite[Proposition 1.5]{CH}, also
see \cite[Proposition 3.5, Theorem 4.6]{MR} for a stronger statement in the local case. The only full dimensional support summand is $IC_{M_d}$. 
By induction on $d\in \mathbb{N}$ and after taking the dual, it suffices to check that 
$IH^{\text{odd}}_c(M_d, \mathbb{Q})=0$. The variety $M_n$ has finite quotient singularities, so it suffices to show that $IH^{\text{odd}}_c(M_d, \mathbb{Q})\cong H^{\text{odd}}_c(M_d, \mathbb{Q})=0$. 
This is true because \[H^{\cdot}_c(M_d, \mathbb{Q})
\cong \left(H^{\cdot}_c(S^d, \mathbb{Q})\right)^{\mathfrak{S}_d}
\cong \left(H^{\cdot}_c(S, \mathbb{Q})^{\otimes d}\right)^{\mathfrak{S}_d}.\]
\end{proof}

Let $X$ and $Y$ be possibly singular varieties. Assume $X$ has an action of a reductive group $G$ and let $\mathfrak{X}=X/G$. Recall the Atiyah-Hirzebruch isomorphism
\[G_i^{\text{top}}(Y)_{\mathbb{Q}}\xrightarrow{\sim} \bigoplus_{j\in\mathbb{Z}} H^{\text{BM}}_{i+2j}(Y, \mathbb{Q}).\]
Using Totaro's approximations for the stack $\mathfrak{X}$, 
there is an inclusion map 
\begin{equation}\label{subsection5310}
    G_i^{\text{top}}(\mathfrak{X})_{\mathbb{Q}}\hookrightarrow \prod_{j\in\mathbb{Z}} H^{\text{BM}}_{i+2j}(\mathfrak{X}, \mathbb{Q}).
\end{equation}
 Using Proposition \ref{pp1} and the Künneth Theorem for $G^{\text{top}}_\cdot$, we obtain:

\begin{cor}\label{ku}
Let $S$ be a surface with $H^1(S, \mathbb{Q})=0$ and let $d, e\in \mathbb{N}$. Then $G_0^{\text{top}}\left(\mathfrak{M}_d\times \mathfrak{M}_e\right)_{\mathbb{Q}}\cong G_0^{\text{top}}\left(\mathfrak{M}_d\right)_{\mathbb{Q}}\otimes G_0^{\text{top}}\left(\mathfrak{M}_e\right)_{\mathbb{Q}}$. 
\end{cor}

\subsubsection{}\label{pbwth}
Using induction, Theorem \ref{thm}, and Corollary \ref{ku}, we see that
\begin{equation}\label{kunn}
    \bigotimes_{i=1}^k K_0'^{\text{top}}\left(\mathbb{M}(d_i)_{w_i}\right)_{\mathbb{Q}}\cong
 K'^{\text{top}}_0(\mathbb{M}_A)_{\mathbb{Q}}
\end{equation}
for $A=(d_i, w_i)_{i=1}^k$ in $V^d_w$ or $U^d_w$.

We define inductively on $(d,w)$ a (split) subspace \[\ell_{d,w}: P(d)_w\hookrightarrow K'^{\text{top}}_0\left(\mathbb{M}(d)_w\right)_{\mathbb{Q}}\] with a surjection
\[\pi_{d,w}: K_0'^{\text{top}}\left(\mathbb{M}(d)_w\right)_{\mathbb{Q}}\twoheadrightarrow P(d)_w\] such that $\pi_{d,w}\ell_{d,w}=\text{id}$.
Let $P(1)_w:=K'^{\text{top}}_0\left(\mathbb{M}(1)_w\right)_{\mathbb{Q}}.$ 
Let $A=(d_i, w_i)_{i=1}^k\in U^d_w$ with $k\geq 2$. Let $\pi_A$ be the natural projection
\[\pi_A:=\otimes_{i=1}^k \pi_{d_i, w_i}: K'^{\text{top}}_0(\mathbb{M}_A)_{\mathbb{Q}}\twoheadrightarrow \otimes_{i=1}^k P(d_i)_{w_i}.\]
Let $A=(d_i,w_i)_{i=1}^k\in U^d_w$ with $A>(d,w)$, or alternatively with $k\geq 2$. Let $P_A:=\otimes_{i=1}^k P(d_i)_{w_i}$ and let $\pi_A$ be the natural projection:
\[\pi_A:=\otimes_{i=1}^k \pi_{d_i, w_i}: K'^{\text{top}}_0(\mathbb{M}_A)_{\mathbb{Q}}\twoheadrightarrow P_A.\]
For $\sigma\in \mathfrak{S}_k$, denote by $\sigma(A)$ the partition $\left(d_{\sigma(i)}, w_{\sigma(i)}\right)_{i=1}^k$.
Let $K_A$ be the kernel of the map
\[\left(\bigoplus_{\sigma\in\mathfrak{S}_{k}}
\pi_{\sigma(A)}\right) \left(\bigoplus_{\sigma\in\mathfrak{S}_{k}}
\sigma\right)\Delta_{A}:
K'^{\text{top}}_0(\mathbb{M}(d)_w)_{\mathbb{Q}}\to \bigoplus_{\sigma\in\mathfrak{S}_{k}}P_{\sigma(A)}.\]
Define
\[P(d)_w:=\bigcap_{A>(d,w)}
 K_A.\]
 Given the construction of these spaces of primitive generators $P(d)_w$, Theorem \ref{thm2} follows formally from Theorem \ref{thm51} exactly as in \cite[Proposition 5.5]{P}, \cite[Theorem 5.13]{P2}.

\end{document}